\documentclass[11pt]{article}

\usepackage{amsthm, amsmath, amssymb, amsfonts, url, booktabs, tikz, setspace, fancyhdr, bm}
\definecolor{seedblue}{RGB}{67,112,156}
\usepackage{cancel}
\usepackage{geometry}
\usepackage{hyperref, enumerate}
\usepackage[shortlabels]{enumitem}
\usepackage[babel]{microtype}
\usepackage[english]{babel}
\usepackage[capitalise]{cleveref}
\usepackage{comment}
\usepackage{bbm}
\usepackage{csquotes}
\usepackage{mathabx}
\usepackage{tikz}
\usepackage{subcaption}
\usepackage{graphicx}
\usepackage{float}
\usepackage{xcolor}
\usepackage[normalem]{ulem}
\usepackage{diagbox}
\usetikzlibrary{positioning, arrows.meta, shapes.geometric}

\counterwithin{figure}{section}

\newtheorem{theorem}{Theorem}[section]

\newtheorem{lemma}[theorem]{Lemma}

\newtheorem{conj}[theorem]{Conjecture}
\newtheorem{claim}[theorem]{Claim}

\newtheorem{fact}[theorem]{Fact}

\theoremstyle{definition}
\newtheorem{defn}[theorem]{Definition}

\newtheorem*{defn-non}{Definition}

\definecolor{rosepink}{RGB}{255,102,204}
\definecolor{dateplum}{HTML}{993366}
\definecolor{darkdateplum}{RGB}{128,0,32}
\definecolor{lightdateplum}{RGB}{219,112,147}
\definecolor{darkred}{RGB}{139,0,0}
\definecolor{lightred}{RGB}{240,130,100}

\newlist{Case}{enumerate}{3}
\setlist[Case, 1]{%
    label           =   {\bfseries Case \arabic*.},
    labelindent=1em ,labelwidth=1cm, labelsep*=1em, leftmargin =!
}
\setlist[Case, 2]{%
    label           =   {\bfseries Subcase \arabic{Casei}.\arabic*.},
    labelindent=-1em ,labelwidth=1cm, labelsep*=1em, leftmargin =!
}
\setlist[Case, 3]{%
    label           =   {\bfseries Subsubcase \arabic{Casei}.\arabic{Caseii}.\arabic*.},
    labelindent=-1em ,labelwidth=1cm, labelsep*=1em, leftmargin =!
}

\newenvironment{poc}{\begin{proof}[Proof of claim]}{\end{proof}}

\usepackage{todonotes}

\title{The asymptotic Erd\H{o}s--Frankl--F\"uredi cover-free conjecture}
\author{
Xinqi Huang\thanks{School of Mathematical Sciences, University of Science and Technology of China, Hefei, China and Extremal Combinatorics and Probability Group (ECOPRO), Institute for Basic Science (IBS), Daejeon, South Korea.
Email: \texttt{huangxq@mail.ustc.edu.cn}. Supported by the USTC Excellent PhD Students Overseas, the Institute for Basic Science (IBS-R029-C4), the National Key Research and Development Programs of China 2023YFA1010200, the NSFC under Grants No. 12171452 and No. 12231014 and Innovation Program for Quantum Science and Technology 2021ZD0302902.
}
\and
Hong Liu\thanks{Extremal Combinatorics and Probability Group (ECOPRO), Institute for Basic Science (IBS), Daejeon, South Korea. Emails: \texttt{hongliu@ibs.re.kr}. Supported by the Institute for Basic Science (IBS-R029-C4).}
}

\begin{document}
\date{}
\maketitle
\begin{abstract}
In 1985, Erd{\H{o}}s, Frankl and F{\"u}redi conjectured that
the smallest $n$ for which an $n$-element set admits an
$r$-cover-free family with more than $n$ members is
$(1+o(1))r^2$ as $r\to\infty$.
We confirm this conjecture by proving the lower bound
$r^2+r+2$ for every $r\ge1$.
\end{abstract}

\section{Introduction}

A family $\mathcal{F}$ of nonempty subsets of $[n]=\{1,\ldots,n\}$ is called \emph{$r$-cover-free} if no member is contained in the union of at most $r$ other members, that is,
for every $A\in \mathcal{F}$ and every $\mathcal{B}\subset \mathcal{F}\setminus\{A\}$ with $|\mathcal{B}|\le r$, one has 
$A\not\subseteq \bigcup_{B\in \mathcal{B}} B$.
This notion goes back to the work of Kautz and Singleton~\cite{KS64}
on superimposed codes.
Define 
$$
    f(n,r) = \max\{|\mathcal{F}|: \mathcal{F}\subset 2^{[n]} \text{ is $r$-cover-free}\}.
$$
The extremal function $f(n,r)$ was studied by Erd{\H{o}}s,
Frankl and F{\"u}redi~\cite{EFF82,EFF85}.
Cover-free families and their generalizations have since received
considerable attention; see, for example,
\cite{AB16,BG17,Furedi96,IM24,Ruszinko94,SW04,SWZ00}.

The family of all singletons shows that $f(n,r)\ge n$.
This raises a natural question: when is the trivial family of all
singletons optimal?
If an $r$-cover-free family on $[n]$ has more than $n$ members,
then adjoining a new element to the ground set and its singleton
to the family gives an $r$-cover-free family on $[n+1]$ with
more than $n+1$ members.
Thus $f(n_0,r)>n_0$ implies $f(n,r)>n$ for every $n\ge n_0$.
This leads to the threshold introduced by Erd{\H{o}}s,
Frankl and F{\"u}redi~\cite{EFF85}, which we denote by
$$T(r)=\min\{n:f(n,r)>n\}.$$

An affine-plane construction gives $T(r)\le (1+o(1))r^2$
as $r\to\infty$~\cite{EFF85}.
Indeed, for any prime power $q>r$, an affine plane of order $q$
has $q^2$ points and $q^2+q$ lines.
Each line contains $q$ points and intersects any other line
in at most one point, so it cannot be covered by $r$ other lines.
Thus its lines form an $r$-cover-free family, giving $T(r)\le q^2$.
By the prime number theorem, we may choose a prime
$q>r$ with $q=(1+o(1))r$, which yields the claimed bound.
In particular, $T(r)\le (r+1)^2$ whenever $r+1$ is a prime power.
Erd{\H{o}}s, Frankl and F{\"u}redi~\cite{EFF85} conjectured
that this upper bound is asymptotically tight.

\begin{conj}\label{conj:EFF}
    As $r\to\infty$, we have $T(r)=(1+o(1))r^2$.
\end{conj}

Earlier bounds on $T(r)$ and results for small values of $r$
can be found in~\cite{CH07,DR83,HH01}.
The best general lower bound was due to
Shangguan and Ge~\cite{SG16}, who proved that
$T(r)\ge \frac{15+\sqrt{33}}{24}r^2\approx 0.86n^2$.
We prove Conjecture~\ref{conj:EFF} by establishing
the following bound.

\begin{theorem}\label{thm:main}
    For every $r\ge 1$, we have $T(r)\ge r^2+r+2$.
\end{theorem}

Cover-free families also have extensive applications in coding theory
and cryptography, including frameproof codes and digital signature
schemes; see, for example,
\cite{Blackburn03,BS98,IM19,LMS25,PWX04,STW00}.

Ge, Shangguan and Wang~\cite{GSW19} studied the minimum length
of a binary $r$-frameproof code with more codewords than its length,
and conjectured that this threshold is $(1+o(1))r^2$
as $r\to\infty$.
Zhao and Zhang~\cite{ZZ24} subsequently formulated the analogous
conjecture for $q$-ary wide-sense $r$-frameproof codes,
for every fixed $q\ge2$.
Both conjectures follow directly from Theorem~\ref{thm:main}
and the relations established in these two papers.

\paragraph{Our method.}
Inspired by the approach developed in our recent joint work
with Chao~\cite{CHL26} on the Brown--Erd\H{o}s--S\'os problem,
our proof uses a double-counting argument for pairs of
elements of the ground set.
A pair is \emph{private} if it is contained in exactly one
member of the family.
The total number of private pairs is therefore at most
$\binom n2$, but counting these pairs alone does not suffice
to obtain the desired bound.
We therefore carefully select additional collections of
non-private pairs, which we call \emph{pseudo-private}.
An auxiliary bipartite graph records the incidences
between members of the family and the pairs selected
for them.
We show that this graph is $C_4$-free and use this property
to obtain a lower bound on the number of distinct
selected pairs.
These additional pairs compensate for the deficits
in private pairs, and incorporating them into the
double-counting argument yields the desired lower bound
on $n$.

\section{Proof of \Cref{thm:main}}
We begin with some definitions and a structural lemma
that will be used in the proof.

\begin{defn}\label{def:family}
    Let $\mathcal H\subseteq2^V$ be a family whose union is $V$.
    Its \emph{covering number} $\kappa(\mathcal H)$ is the
    minimum number of members of $\mathcal H$ whose union is $V$.
    For $U\subseteq V$, write
    $\mathcal H[U]=\{Q\in\mathcal H:Q\subseteq U\}$
    for the subfamily induced by $U$.
    Let $G(\mathcal H)$ be the graph on $V$ with edge set
    $\mathcal H\cap\binom V2$.

    The family $\mathcal H$ is a \emph{down-set} if
    $Q\in\mathcal H$ and $Q'\subseteq Q$ imply
    $Q'\in\mathcal H$.
    A down-set $\mathcal H$ is \emph{critical} if
    $\kappa(\mathcal H[V\setminus\{v\}])
    =\kappa(\mathcal H)-1$ for every $v\in V$.
\end{defn}

We use the following dual form of Stehl\'ik's theorem~\cite{Stehlik06}.

\begin{lemma}[Stehl\'ik~\cite{Stehlik06}]\label{lem:stehlik}
    Let $\mathcal H\subseteq2^V$ be a critical down-set
    whose union is $V$, and write $t=\kappa(\mathcal H)$.
    If $G(\mathcal H)$ is connected, then for every $z\in V$,
    there is a partition
    $ V\setminus\{z\} =W_1\mathbin{\dot\cup}\cdots\mathbin{\dot\cup}W_{t-1}$,
    where $W_i\in\mathcal H$ and $|W_i|\ge2$
    for every $1\le i\le t-1$.
    In particular, $|V|\ge2t-1$.
\end{lemma}

Throughout this section, we argue 
by induction on $r$. 
For $r=1$, a cover-free
family is an antichain, and on at most three elements it has at most
as many members as elements. Thus $T(1)\ge4$.
Now let $r\ge2$, suppose for a contradiction that
$n=T(r)\le r^2+r+1$, and choose an $r$-cover-free family
$\mathcal F\subseteq2^{[n]}$ with $|\mathcal F|=n+1$.
The minimality of $n$ ensures that $\bigcup_{A\in\mathcal F}A=[n]$.
An element of $[n]$ is \emph{private} if it belongs to exactly one
member of $\mathcal F$.

\begin{fact}\label{fact:preprocessing}
    Minimality of $n$ and induction on $r$ imply that:
    \begin{enumerate}[(1)]
        \item\label{item:nonprivate}
        No element of $[n]$ is private.
        \item\label{item:setsize}
        Every member $A\in\mathcal F$ satisfies
        $r+1\le |A|\le 2r-1$.
    \end{enumerate}
\end{fact}

\begin{proof}
    Suppose that some element $x\in[n]$ belongs to exactly one member $A\in\mathcal F$. 
    Then $\mathcal F'=\mathcal F\setminus\{A\}$
    is still an $r$-cover-free family over $[n]\setminus\{x\}$ satisfying
    $|\mathcal F'|=|\mathcal F|-1>n-1$.
    This contradicts the definition of $n=T(r)$. 
    This proves~\ref{item:nonprivate}.
    Consequently, every member $A\in\mathcal F$ has size at least $r+1$.
    Indeed, otherwise, for each $x\in A$, we could choose another member containing $x$, thereby covering $A$ with at most $|A|\le r$ other members, a contradiction.

    For the upper bound, 
    suppose that $|A|\ge2r$ for some $A\in\mathcal F$, and let
    $\mathcal F_A
        =\{B\setminus A:B\in\mathcal F\setminus\{A\}\}$.
    The members of $\mathcal F_A$ are nonempty and pairwise
    distinct, since otherwise some member of $\mathcal F$
    would be covered by at most two others.
    Moreover, it is easy to check that $\mathcal F_A$ is an $(r-1)$-cover-free family on
    $n-|A|$ elements with
    $|\mathcal F_A|=|\mathcal F|-1= n>n-|A|$.
    By the induction hypothesis on $r$, we have
    $n-|A|\ge(r-1)^2+(r-1)+2=r^2-r+2$.
    Since $|A|\ge2r$, this yields $n\ge r^2+r+2$, contrary to our assumption.
    Thus every member of $\mathcal F$ has size at most $2r-1$, completing the proof of~\ref{item:setsize}.
\end{proof}

For each $A\in\mathcal F$, let
$\mathcal G_A=\{B\cap A:B\in\mathcal F\setminus\{A\}\}$,
and denote its downward closure by
$\mathcal H_A=\{Q\subseteq A:Q\subseteq Q_1
\text{ for some }Q_1\in\mathcal G_A\}$.
By Fact~\ref{fact:preprocessing} \ref{item:nonprivate}, every element of $A$ belongs to
some member of $\mathcal G_A$.
Thus $\bigcup_{Q\in\mathcal G_A}Q=A$, and
$\mathcal H_A$ contains every singleton of $A$.
Since $\mathcal F$ is $r$-cover-free, no collection of
at most $r$ members of $\mathcal G_A$ covers $A$.
Moreover, taking the downward closure does not change
the covering number, since each member of $\mathcal H_A$
can be enlarged to a member of $\mathcal G_A$.
Therefore, $\kappa(\mathcal H_A)=\kappa(\mathcal G_A)\ge r+1$.

A pair $e\in\binom A2$ is called \emph{private to $A$}
if it is contained in no other member of $\mathcal F$.
Equivalently, $e$ is a nonedge of $G(\mathcal H_A)$.
Let $q_A$ denote the number of private pairs of $A$;
thus $q_A=\binom{|A|}{2}-e(G(\mathcal H_A))$.
We call $A$ \emph{good} if $q_A\ge\binom{r+1}{2}$,
and \emph{bad} otherwise.
Let $\mathcal B$ denote the family of bad members.

Our strategy is to obtain a lower bound on the number of
distinct pairs contained in members of $\mathcal F$.
Since there are at most $\binom n2$ such pairs, this will
yield a lower bound on $n$.
Private pairs belonging to different members of $\mathcal F$
are distinct, so they contribute
$\sum_{A\in\mathcal F}q_A$ pairs.
If every member of $\mathcal F$ were good, this would give
$\binom n2\ge|\mathcal F|\binom{r+1}{2}$ and hence
the desired bound on $n$.
To compensate for the deficits of the bad members,
we will also count collections of non-private pairs
whose overlaps can be controlled using the following condition.

\begin{defn}\label{def:pseudo}
    Let $\mathcal H\subseteq2^V$ be a family.
    A collection $\mathcal P\subseteq E(G(\mathcal H))$
    is \emph{pseudo-private} if
    $|\mathcal P\cap\binom Q2|\le1$
    for every $Q\in\mathcal H$.
\end{defn}

For $\mathcal H=\mathcal H_A$, this condition means that
every member $B\in\mathcal F\setminus\{A\}$ contains
at most one pair from $\mathcal P$.
In particular, if $\mathcal P_A$ and $\mathcal P_B$
are pseudo-private collections in $\mathcal H_A$
and $\mathcal H_B$, respectively, then
$|\mathcal P_A\cap\mathcal P_B|\le1$ whenever $A\ne B$.

The following lemma shows that we can choose pseudo-private
collections for the bad members so that their union,
together with the private pairs, gives the required count.

\begin{lemma}\label{lem:assignment}
    There exists an assignment of a pseudo-private
    collection $\mathcal P_A$ in $\mathcal H_A$
    to each $A\in\mathcal B$ such that
    \[
        \sum_{A\in\mathcal F}q_A
        +\left|\bigcup_{A\in\mathcal B}\mathcal P_A\right|
        \ge |\mathcal F|\binom{r+1}{2}.
    \]
\end{lemma}

\begin{proof}[Deduction of \Cref{thm:main} from Lemma~\ref{lem:assignment}]
Private pairs belonging to different members of $\mathcal F$
are distinct, and every pair in
$\bigcup_{A\in\mathcal B}\mathcal P_A$ is non-private.
Thus Lemma~\ref{lem:assignment} and $|\mathcal F|\ge n+1$ give
\[
    \binom n2
    \ge \sum_{A\in\mathcal F}q_A
        +\left|\bigcup_{A\in\mathcal B}\mathcal P_A\right|
    \ge (n+1)\binom{r+1}{2}.
\]
This is impossible when $n\le r^2+r+1$, proving the theorem.
\end{proof}

It remains to prove Lemma~\ref{lem:assignment}.
We first establish a lower bound on the size of a
pseudo-private collection associated with each bad member.

\begin{lemma}\label{lem:local}
    For every $A\in\mathcal B$, there exists a pseudo-private
    collection $\mathcal P_A$ in $\mathcal H_A$ such that
    \[
        |\mathcal P_A|
        \ge 2\left(\binom{r+1}{2}-q_A\right)
            +\frac{r(r+1)+1}{6}.
    \]
\end{lemma}

\begin{proof}[Proof of Lemma~\ref{lem:assignment}
assuming Lemma~\ref{lem:local}]
Write $h=|\mathcal B|$.
For each $A\in\mathcal B$, let
$d_A=\binom{r+1}{2}-q_A$,
and choose a pseudo-private collection $\mathcal P_A$
as in Lemma~\ref{lem:local}.
Set $\mathcal D=\bigcup_{A\in\mathcal B}\mathcal P_A$.
For each $e\in\mathcal D$, let
$\mu(e)=|\{A\in\mathcal B:e\in\mathcal P_A\}|$.

We construct an auxiliary bipartite graph $F$ with vertex
classes $\mathcal B$ and $\mathcal D$, where $B\in\mathcal B$
is adjacent to $e\in\mathcal D$ if and only if
$e\in\mathcal P_B$.
We claim that $F$ is $C_4$-free.
Indeed, a $4$-cycle would give distinct
$B_1,B_2\in\mathcal B$ and distinct $e_1,e_2\in\mathcal D$
such that $e_1,e_2\in\mathcal P_{B_1}\cap\mathcal P_{B_2}$,
contradicting $|\mathcal P_{B_1}\cap\mathcal P_{B_2}|\le1$.
By double counting the pairs $(\{B_1,B_2\},e)$,
where $B_1\ne B_2$ are both adjacent to $e$ in $F$,
we obtain
\[
    \sum_{e\in\mathcal D}\binom{\mu(e)}{2}
    \le \binom h2.
\]
Using the inequality
$x\le 2+\frac13\binom x2$
for every nonnegative integer $x$
and Lemma~\ref{lem:local}, we obtain
\begin{align*}
    \sum_{A\in\mathcal B}
        \left(2d_A+\frac{r(r+1)+1}6\right)
    \le \sum_{A\in\mathcal B}|\mathcal P_A|
     = \sum_{e\in\mathcal D}\mu(e)
    \le 2|\mathcal D|
        +\frac13\sum_{e\in\mathcal D}\binom{\mu(e)}2
    \le 2|\mathcal D|+\frac13\binom h2.
\end{align*}
Consequently,
\[
    |\mathcal D|
    \ge \sum_{A\in\mathcal B}d_A
        +\frac{h}{12}\bigl(r(r+1)+2-h\bigr)
    \ge \sum_{A\in\mathcal B}d_A,
\]
where the last inequality follows from
$h\le|\mathcal F|=n+1\le r(r+1)+2$.

Since every good member has at least
$\binom{r+1}{2}$ private pairs, we conclude that
\begin{align*}
    \sum_{A\in\mathcal F}q_A+|\mathcal D|
    \ge 
    \sum_{A\in\mathcal F}q_A
        +\sum_{A\in\mathcal B}d_A
    =
    \sum_{A\in\mathcal F\setminus\mathcal B}q_A
        +h\binom{r+1}{2}
    \ge 
    |\mathcal F|\binom{r+1}{2},
\end{align*}
as required.
\end{proof}


So it remains to prove Lemma~\ref{lem:local}.

\begin{proof}[Proof of Lemma~\ref{lem:local}]
    Fix a bad member $A\in\mathcal F$.
    Choose an inclusion-minimal subset $V\subseteq A$ such that $\kappa(\mathcal H_A[V])\ge r+1$, and put $\mathcal H=\mathcal H_A[V]$.
    Since $\mathcal H$ contains every singleton of $V$, deleting one vertex decreases its covering number by at most one.
    The minimality of $V$ therefore implies that $\kappa(\mathcal H)=r+1$ and $\kappa(\mathcal H[V\setminus\{v\}])=r$ for every $v\in V$. 
    Thus $\mathcal H$ is critical.

    Let $q=\binom{|V|}{2}-e(G(\mathcal H))$.
    Then $q\le q_A<\binom{r+1}{2}$.
    Moreover, every pseudo-private collection in $\mathcal H$ is also pseudo-private in $\mathcal H_A$, since $Q\cap V\in\mathcal H$ for every $Q\in\mathcal H_A$.
    It therefore suffices to construct a pseudo-private collection $\mathcal P$ in $\mathcal H$ such that
    \begin{equation}\label{eq:local-target}
        |\mathcal P|
        \ge 
        2\left(\binom{r+1}{2}-q\right) + \frac{r(r+1)+1}6.
    \end{equation}

    Write $|V|=r+s$, where $1\le s\le r-1$, and let $V_1,\ldots,V_c$ be the vertex sets of the connected components of $G(\mathcal H)$.
    For each $i\in[c]$, put $\mathcal H_i=\mathcal H[V_i]$, $t_i=\kappa(\mathcal H_i)$, and $w_i=|V_i|$.
    Every nonempty member of $\mathcal H$ lies in one component, since all of its pairs are edges of $G(\mathcal H)$.
    Consequently, $\sum_{i=1}^c t_i=r+1$.
    The criticality of $\mathcal H$ then implies that every $\mathcal H_i$ is critical.

    We now make the following construction in each component.
    For each $i\in[c]$, let $G_i=G(\mathcal H_i)$,
    $M_i=w_i-2t_i+1$, and $q_i=\binom{w_i}{2}-e(G_i)$.
    Choose a vertex $z_i$ of maximum degree in $G_i$.
    By Lemma~\ref{lem:stehlik}, there is a partition
    \[
        V_i\setminus\{z_i\}
        =
        W_{i,1}\mathbin{\dot\cup}\cdots
        \mathbin{\dot\cup}W_{i,t_i-1},
        \qquad
        W_{i,j}\in\mathcal H_i,\quad |W_{i,j}|\ge2
        \text{ for every }j\in[t_i-1].
    \]
    In particular, $M_i=\sum_{j=1}^{t_i-1}(|W_{i,j}|-2)\ge0$.

    Let $\ell_i$ be the number of parts $W_{i,j}$ of size at least three, and let $W_i$ be their union.
    Set $a_i=|W_i|$ and $U_i=V_i\setminus W_i$.
    Since $\ell_i\le M_i$, we have
    \[
        a_i=M_i+2\ell_i\le3M_i,
        \qquad
        |U_i|
        =
        w_i-M_i-2\ell_i=2(t_i-\ell_i)-1.
    \]
    Also, $a_i\le w_i-1$, since $z_i\notin W_i$.
    The parts of size two form a perfect matching $\mathcal M_i$ in $G_i[U_i\setminus\{z_i\}]$.

    The matching edges together with $\{z_i\}$ cover $U_i$, so $\kappa(\mathcal H_i[U_i])\le t_i-\ell_i$.
    In fact, equality holds.
    Otherwise, a cover of $U_i$ by fewer than $t_i-\ell_i$ members of $\mathcal H_i[U_i]$, together with the $\ell_i$ deleted parts, would cover $V_i$ with fewer than $t_i$ members of $\mathcal H_i$, a contradiction.

We shall construct a pseudo-private collection in each
restricted family $\mathcal H_i[U_i]$.
Such a collection is also pseudo-private in $\mathcal H_i$,
since $Q\cap U_i\in\mathcal H_i[U_i]$
for every $Q\in\mathcal H_i$.


    For each $i\in[c]$, write $G_i'=G_i[U_i]$.
Let $f_i:U_i\setminus\{z_i\}\to U_i\setminus\{z_i\}$
exchange the endpoints of each edge of $\mathcal M_i$,
and put $N_i=N_{G_i'}(z_i)$.
For a pair $\alpha=\{x,y\}\in\binom{U_i\setminus\{z_i\}}2$,
write $f_i(\alpha)=\{f_i(x),f_i(y)\}$, and define
\[
    J_i=\{\alpha\in E(G_i'-z_i):
                f_i(\alpha)\in E(G_i'-z_i)\}.
\]
In particular, every edge of $\mathcal M_i$ belongs to $J_i$,
since it is fixed by $f_i$ (see \Cref{fig:component-decomposition} (a)).
The following claim will allow us to extract
a pseudo-private subcollection of $J_i$.


\begin{figure}[htbp]
    \noindent
    \begin{subfigure}[t]{0.60\linewidth}
        \centering
        \vspace{0pt}
        \resizebox{\linewidth}{!}{%
        \begin{tikzpicture}[
            font=\small,
            vertex/.style={
                circle,fill=black,
                inner sep=0pt,minimum size=3.2pt
            },
            matching/.style={black,line width=1pt},
            paired/.style={blue!70!black,line width=1pt},
            block/.style={
                draw=black!35,dashed,
                fill=white,line width=.5pt
            }
        ]
            \draw[rounded corners=7pt,black!60]
                (-.25,-.2) rectangle (12.75,5.95);
            \node[fill=white,inner xsep=6pt]
                at (6.25,5.95) {$V_i$};

            \draw[
                rounded corners=5pt,
                draw=black!30,fill=black!3
            ] (0,0) rectangle (4.9,5.55);

            \draw[
                rounded corners=5pt,
                draw=black!30,fill=blue!2
            ] (5.25,0) rectangle (12.5,5.55);

            \node at (2.45,5.15) {$W_i$};
            \node at (8.875,5.15) {$U_i$};

            \begin{scope}[shift={(1.35,3.95)}]
                \draw[block] (0,0) circle (.86);
                \foreach \j/\angle in {1/90,2/210,3/330}
                    \node[vertex] (a\j) at (\angle:.57) {};
                \foreach \u/\v in {1/2,1/3,2/3}
                    \draw[black!65] (a\u)--(a\v);
            \end{scope}

            \begin{scope}[shift={(3.5,3.95)}]
                \draw[block] (0,0) circle (.86);
                \foreach \j/\angle in
                    {1/45,2/135,3/225,4/315}
                    \node[vertex] (b\j) at (\angle:.62) {};
                \foreach \u/\v in
                    {1/2,1/3,1/4,2/3,2/4,3/4}
                    \draw[black!65] (b\u)--(b\v);
            \end{scope}

            \begin{scope}[shift={(2.45,1.9)}]
                \draw[block] (0,0) circle (1.02);
                \foreach \j/\angle in
                    {1/90,2/162,3/234,4/306,5/378}
                    \node[vertex] (c\j) at (\angle:.76) {};
                \foreach \u/\v in
                    {1/2,1/3,1/4,1/5,2/3,
                     2/4,2/5,3/4,3/5,4/5}
                    \draw[black!65] (c\u)--(c\v);
            \end{scope}

            \node at (2.45,.4)
                {Blocks $W_{i,j}$ with $|W_{i,j}|\ge3$};

            \node[vertex,label=above:{$z_i$}]
                (z) at (9,4.45) {};

            \node[vertex,label=left:{$x$}]
                (x) at (7.55,3.4) {};
            \node[vertex,label=right:{$f_i(x)$}]
                (fx) at (10.55,3.4) {};

            \node[vertex,label=left:{$y_1$}]
                (y1) at (7.55,2.25) {};
            \node[vertex,label=right:{$f_i(y_1)$}]
                (fy1) at (10.55,2.25) {};

            \node[vertex,label=left:{$y_2$}]
                (y2) at (7.55,1.0) {};
            \node[vertex,label=right:{$f_i(y_2)$}]
                (fy2) at (10.55,1.0) {};

            \draw[matching] (x)--(fx);
            \draw[matching] (y1)--(fy1);
            \draw[matching] (y2)--(fy2);

            \draw[paired]
                (y1)--node[
                    midway,left=7pt,text=blue!70!black
                ] {$y_1y_2\in J_i$} (y2);
            \draw[paired] (fy1)--(fy2);

            \node at (9.05,.4)
                {Matching edges in $\mathcal M_i$};
        \end{tikzpicture}%
        }
        \caption{The decomposition of $V_i$.}
        \label{fig:component-decomposition-a}
    \end{subfigure}\hfill
    \begin{subfigure}[t]{0.4\linewidth}
    \centering
    \vspace{0pt}
    \resizebox{0.96\linewidth}{!}{%
    \begin{tikzpicture}[
        font=\small,
        vertex/.style={
            circle,fill=black,
            inner sep=0pt,minimum size=3.5pt
        },
        old matching/.style={
            draw=black!45,densely dashed,line width=.8pt
        },
        cover edge/.style={
            line width=1.35pt,line cap=round
        }
    ]
        \draw[rounded corners=7pt,black!60]
            (0,0) rectangle (9.35,6.925);
        \node[fill=white,inner xsep=6pt]
            at (4.675,6.925) {$U_i$};

        \coordinate (x)  at (3.3,5.6);
        \coordinate (y)  at (3.3,3.1);
        \coordinate (fx) at (1.05,5.6);
        \coordinate (fy) at (1.05,3.1);
        \coordinate (v)  at (5.1,4.35);
        \coordinate (fv) at (7.65,4.35);
        \coordinate (u)  at (5.1,1.35);
        \coordinate (fu) at (7.65,1.35);
        \coordinate (z)  at (3.3,1.35);

        \draw[
            draw=black!45,
            dashed,
            fill=black!3,
            line width=.7pt
        ] (u) circle (.6);
        \node[anchor=south] at (5.1,2.05) {$N_i$};

        \draw[old matching] (x)--(fx);
        \draw[old matching] (y)--(fy);
        \draw[old matching] (v)--(fv);
        \draw[old matching] (u)--(fu);

        \path[
            draw=red!75!black,
            fill=red!9,
            line width=1.35pt,
            line join=round
        ] (x)--(y)--(v)--cycle;

        \draw[cover edge,green!45!black]
            (fx)--(fy);

        \draw[cover edge,blue!75!black]
            (fv)--(fu);

        \draw[cover edge,violet!85!black]
            (z)--(u);

        \node[vertex,label=above:{$x$}] at (x) {};
        \node[vertex,label=below:{$y$}] at (y) {};
        \node[vertex,label=above:{$f_i(x)$}] at (fx) {};
        \node[vertex,label=below:{$f_i(y)$}] at (fy) {};
        \node[vertex,label=above:{$v$}] at (v) {};
        \node[vertex,label=right:{$f_i(v)$}] at (fv) {};
        \node[vertex,label=below:{$u$}] at (u) {};
        \node[vertex,label=right:{$f_i(u)$}] at (fu) {};
        \node[vertex,label=left:{$z_i$}] at (z) {};

        \node at (4.675,.3)
            {$N_i=N_{G_i'}(z_i)$};
    \end{tikzpicture}%
    }
    \caption{A smaller cover when $y\ne f_i(x)$.}
    \label{fig:trace-cover}
\end{subfigure}

    \caption{The local construction.
        Only selected edges are shown.}
    \label{fig:component-decomposition}
\end{figure}
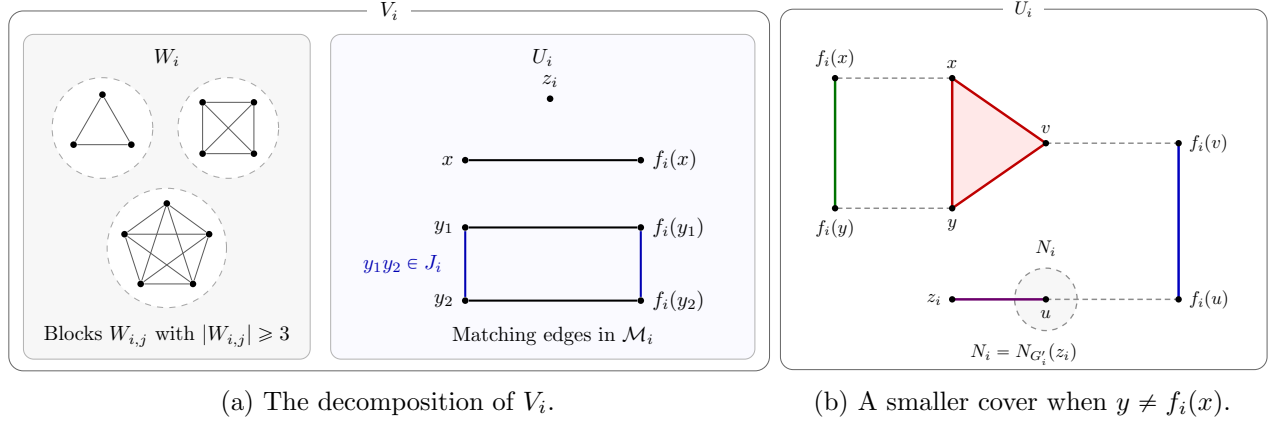


\begin{claim}\label{cl:trace}
    For every $i\in[c]$, suppose that $\{x,y\}\in J_i$,
    $v\in U_i\setminus\{z_i,x,y,f_i(x),f_i(y)\}$, and
    $\{x,y,v\}\in\mathcal H_i[U_i]$.
    Then
    \[
        z_i f_i(v)\notin E(G_i'),
        \qquad
        |N_{G_i'}(f_i(v))\cap f_i(N_i)|\le2.
    \]
\end{claim}

\begin{poc}
Fix $i\in[c]$.
Recall that $\kappa(\mathcal H_i[U_i])=t_i-\ell_i$
and that $\mathcal M_i$ is a perfect matching of
$U_i\setminus\{z_i\}$.
Suppose first that $z_i f_i(v)\in E(G_i')$.
If $y\ne f_i(x)$, then the sets
\[
    \{x,y,v\},\qquad
    \{f_i(x),f_i(y)\},\qquad
    \{z_i,f_i(v)\},
\]
together with the matching edges disjoint from these sets,
cover $U_i$ with $t_i-\ell_i-1$ members of
$\mathcal H_i[U_i]$, a contradiction.
If $y=f_i(x)$, omitting $\{f_i(x),f_i(y)\}$
gives a cover of the same size.
Thus $z_i f_i(v)\notin E(G_i')$.

We next show that
$N_{G_i'}(f_i(v))\cap f_i(N_i)
    \subseteq \{x,y,f_i(x),f_i(y)\}$.
Suppose otherwise that $p$ belongs to the left-hand side
but not the right-hand side, and write $p=f_i(u)$
with $u\in N_i$.
Since $p$ is a neighbor of $f_i(v)$, we have $u\ne v$.
Moreover, the first conclusion gives $f_i(v)\notin N_i$,
so $u\ne f_i(v)$.
It follows that the matching edge $\{u,f_i(u)\}$
is disjoint from the matching edges containing $x,y,v$.
Now the sets
\[
    \{x,y,v\},\qquad
    \{f_i(x),f_i(y)\},\qquad
    \{f_i(v),f_i(u)\},\qquad
    \{z_i,u\},
\]
together with the remaining matching edges, cover $U_i$
with $t_i-\ell_i-1$ members of $\mathcal H_i[U_i]$ (see \Cref{fig:component-decomposition} (b)).
When $y=f_i(x)$, omit $\{f_i(x),f_i(y)\}$.
This contradiction proves the inclusion.

If $y=f_i(x)$, the inclusion immediately gives
$|N_{G_i'}(f_i(v))\cap f_i(N_i)|\le2$.
Otherwise, the four vertices $x,y,f_i(x),f_i(y)$
are distinct.
The intersection cannot contain both $x$ and $f_i(y)$:
these two memberships imply
$z_i f_i(x),f_i(v)f_i(y)\in E(G_i')$, and hence
\[
    \{x,y,v\},\qquad
    \{z_i,f_i(x)\},\qquad
    \{f_i(v),f_i(y)\},
\]
together with the remaining matching edges, cover $U_i$
with $t_i-\ell_i-1$ members.
Similarly, the intersection cannot contain both
$y$ and $f_i(x)$.
It therefore contains at most one vertex from each of
$\{x,f_i(y)\}$ and $\{y,f_i(x)\}$, proving the desired bound.
\end{poc}

\begin{claim}\label{cl:local-component}
    For every $i\in[c]$, there exists a pseudo-private
    collection $\mathcal P_i$ in $\mathcal H_i[U_i]$
    such that
    \[
        |\mathcal P_i|
        \ge \binom{|U_i|-1}{2}-2q_i.
    \]
\end{claim}

\begin{poc}
Fix $i\in[c]$, and define
\[
    L_i=\left\{
        v\in U_i\setminus\{z_i\}:
        z_i f_i(v)\notin E(G_i'),\
        |N_{G_i'}(f_i(v))\cap f_i(N_i)|\le2
    \right\}.
\]
We take
$\mathcal P_i=\{\alpha\in J_i:\alpha\cap L_i=\varnothing\}$.

We first verify that $\mathcal P_i$ is pseudo-private.
No member of $\mathcal H_i[U_i]$ can contain two distinct
edges of $\mathcal M_i$: replacing those two edges by
that member in the cover consisting of $\mathcal M_i$
and $\{z_i\}$ would give a cover with fewer than
$t_i-\ell_i$ members.

Suppose now that some $Q\in\mathcal H_i[U_i]$ contains
two distinct pairs $\alpha=\{x,y\}$ and $\beta$
from $\mathcal P_i$.
If $\beta$ has an endpoint $v$ outside
$\{x,y,f_i(x),f_i(y)\}$, then
$\{x,y,v\}\in\mathcal H_i[U_i]$.
Claim~\ref{cl:trace} gives $v\in L_i$,
contrary to $\beta\in\mathcal P_i$.
Otherwise, $\alpha\cup\beta$ contains either two entire
matching edges, which was just excluded, or three
vertices including one entire matching edge.
In the latter case, apply Claim~\ref{cl:trace}
to that matching edge and the third vertex.
This vertex must belong to $L_i$, again contradicting
$\alpha,\beta\in\mathcal P_i$.
Thus $\mathcal P_i$ is pseudo-private.

It remains to estimate $|\mathcal P_i|$.
Put $B_i=U_i\setminus(N_i\cup\{z_i\})$ and
$q_i'=\binom{|U_i|}{2}-e(G_i')$.
Then $q_i'\le q_i$, and $f_i(L_i)\subseteq B_i$,
so $|L_i|\le|B_i|$.
Let $g_i$ be the number of pairs
$\alpha\in\binom{U_i\setminus\{z_i\}}2$
for which both $\alpha$ and $f_i(\alpha)$
are missing from $G_i'$.
There are $q_i'-|B_i|$ missing pairs
in $U_i\setminus\{z_i\}$.
Since $f_i$ is a bijection on these pairs,
inclusion--exclusion gives
\begin{equation}\label{eq:local-J}
    |J_i|
    =\binom{|U_i|-1}{2}-2q_i'+2|B_i|+g_i.
\end{equation}

We next bound the number of pairs deleted from $J_i$.
For $v\in L_i$, let $b_i(v)=|W_i|-d_{G_i}(v,W_i)$ be the number
of its nonneighbors in $W_i$, where $d_G(v,S)=|N_G(v)\cap S|$.
Since $z_i$ has maximum degree in $G_i$,
$d_{G_i'}(v)\le|N_i|+b_i(v)$.
Separating the neighbors of $v$ into $N_i$, $B_i$ and $\{z_i\}$ gives
\[
    d_{G_i'}(v,B_i)
    \le |\{u\in N_i\setminus\{v\}:vu\notin E(G_i')\}|+b_i(v).
\]
For $v\in L_i$, define
\[
    g_i(v)
    =\left|\left\{
        u\in N_i\setminus\{v\}:
        vu\notin E(G_i'),\
        f_i(v)f_i(u)\notin E(G_i')
    \right\}\right|.
\]
Write $d_{J_i}(v)$ for the number of pairs in $J_i$
containing $v$.
Adding the number of pairs in $J_i$ joining $v$ to $N_i$
to the preceding bound, and separating the vertices
$u\in N_i\setminus\{v\}$ according to whether
$f_i(v)f_i(u)$ is an edge, gives
\begin{equation}\label{eq:local-degree}
    d_{J_i}(v)
    \le g_i(v)
       +|N_{G_i'}(f_i(v))\cap f_i(N_i)|
       +b_i(v)
    \le g_i(v)+2+b_i(v).
\end{equation}
Furthermore,
\begin{equation}\label{eq:local-g}
    \sum_{v\in L_i}g_i(v)\le g_i.
\end{equation}
Indeed, pair each missing pair counted by $g_i$
with its image under $f_i$.
Each resulting orbit has two elements,
since every pair fixed by $f_i$ is a matching edge
and hence is present.
A pair contributes at most twice to the sum on the left.
If it contributes twice, both endpoints lie in
$L_i\cap N_i$, so both endpoints of its image lie in
$B_i$ and that image contributes nothing.
Thus each orbit contributes at most two,
exactly its contribution to $g_i$.

The missing pairs between $L_i$ and $W_i$ are counted by $q_i$ but
not by $q_i'$, so
\begin{equation}\label{eq:local-delta}
    \sum_{v\in L_i}b_i(v)\le q_i-q_i'.
\end{equation}
Combining \eqref{eq:local-J}--\eqref{eq:local-delta}
and using $|L_i|\le|B_i|$, we conclude that
\begin{align*}
    |\mathcal P_i|
    \ge |J_i|-\sum_{v\in L_i}d_{J_i}(v)
    \ge \binom{|U_i|-1}{2}-2q_i'
         -\sum_{v\in L_i}b_i(v)
    \ge \binom{|U_i|-1}{2}-q_i-q_i'
    \ge \binom{|U_i|-1}{2}-2q_i,
\end{align*}
as required.
\end{poc}

We now combine the collections given by
Claim~\ref{cl:local-component}.
For each $i\in[c]$, choose such a collection $\mathcal P_i$,
and put $\mathcal P=\bigcup_{i=1}^c\mathcal P_i$.
Each $\mathcal P_i$ is also pseudo-private in $\mathcal H_i$.
Since every nonempty member of $\mathcal H$ lies in one
component, $\mathcal P$ is pseudo-private in $\mathcal H$.
Moreover, all pairs between distinct components are missing,
so $q=\sum_{i=1}^c q_i+\sum_{i<j}w_iw_j$.

Write $d=\binom{r+1}{2}-q>0$ and
$D=\binom{r+s}{2}-\binom{r+1}{2}$, and put
$M=\sum_i M_i$, $a=\sum_i a_i$, and
$X=\sum_i(w_i-1)$.
Summing $w_i=2t_i-1+M_i$ gives
$c=r-s+2+M$, and hence
$X=r+s-c=2s-2-M$.
Put $x_i=w_i-1$ and $y_i=|U_i|-1=x_i-a_i$.
Then $0\le y_i\le x_i$, $\sum_i x_i=X$ and $\sum_i y_i=X-a$.
Claim~\ref{cl:local-component} gives
\begin{equation}
    |\mathcal P|-2d\ge 2D-\sum_i\left(2\binom{x_i+1}{2}-\binom{y_i}{2}\right)\ge 2D-2\binom{X+1}{2}+\binom{X-a}{2}.
       \label{eq:local-combined}
\end{equation}
The second inequality merges the components: the difference between
$2\binom{X+1}{2}-\binom{X-a}{2}$ and the corresponding sum is
$\sum_{i<j}(2x_ix_j-y_iy_j)\ge0$.
We also have $\sum_i\binom{w_i}{2}\le\binom{X+1}{2}$.

The assumption $q<\binom{r+1}{2}$ gives
\begin{equation}\label{eq:local-bad-range}
    D<e(G(\mathcal H))
      \le\sum_i\binom{w_i}{2}
      \le\binom{X+1}{2}
      \le\binom{2s-1}{2}.
\end{equation}
This rules out $s=1$.
Since $D=(s-1)(2r+s)/2$, division by $s-1$ gives
$s>2(r+1)/3$.
Together with $s\le r-1$, this implies $r\ge6$.

We also have $3M\le X$.
Otherwise
$X=2s-2-M<3(s-1)/2$, and, since $r\ge s+1$,
\[
    \binom{X+1}{2}
    <\frac{3(s-1)(3s-1)}8
    <\frac{(s-1)(3s+2)}2\le D,
\]
contrary to \eqref{eq:local-bad-range}.
Since $a\le3M\le X$, the last term of
\eqref{eq:local-combined} is at least $\binom{X-3M}{2}$.
Consequently,
\begin{align*}
    |\mathcal P|-2d
    &\ge 2D-2\binom{X+1}{2}+\binom{X-3M}{2}\\
    &=2rs-s^2-2r+1+7M^2-(4s-7)M\\
    &\ge 2rs-\frac{11}{7}s^2-2r+2s-\frac34,
\end{align*}
where the last inequality follows by completing
the square in $M$.

Finally, $s\le r-1$, $s>2(r+1)/3$, and $r\ge6$
imply that $r-s\ge0$ and $11s-3r-14>0$.
Thus the last expression equals
\[
    \frac{3r^2}{7}
    +\frac{(r-s)(11s-3r-14)}7-\frac34
    \ge \frac{3r^2}{7}-\frac34
    \ge \frac{r(r+1)+1}6,
\]
where the last inequality holds for $r\ge6$.
We conclude that
$|\mathcal P|\ge2d+(r(r+1)+1)/6$,
which proves \eqref{eq:local-target}.
Taking $\mathcal P_A=\mathcal P$ completes the proof.
\end{proof}

\section{Concluding remarks}

Erd{\H{o}}s, Frankl and F{\"u}redi~\cite{EFF85} also proposed
the following stronger conjecture, which is not settled by our result.

\begin{conj}\label{conj:EFF-strong}
    For every $r\ge1$, we have $T(r)\ge(r+1)^2$.
\end{conj}

Together with the affine-plane construction, this would imply
that $T(r)=(r+1)^2$ whenever $r+1$ is a prime power,
establishing the optimality of that construction.

\paragraph{AI usage.}
All mathematical ideas and proofs are due to the authors. We acknowledge the use of AI tools for literature searches and language editing. During the proofreading of an earlier draft, AI also suggested a simplification of part of the calculations in \cref{lem:local}.

\bibliographystyle{abbrv}
\bibliography{cover_free_references}

\end{document}